\documentclass[preprint,12pt]{elsarticle}

\usepackage{hyperref}%
\usepackage{graphicx}%

\usepackage{amsmath}
\usepackage{amsfonts}
\usepackage{amssymb}
\newtheorem{theorem}{Theorem}[section]

\newtheorem{corollary}[theorem]{Corollary}

\newtheorem{lemma}[theorem]{Lemma}

\newtheorem{remark}{Remark}[section]
\newproof{proof}{Proof}

\newcommand{\Eqlab}[1]{\label{eq:#1}}%
\newcommand{\Eqref}[1]{\eqref{eq:#1}}%
\def\bl{\bigl}%
\def\br{\bigr}%
\def\Bl{\Bigl}%
\def\Br{\Bigr}%
\def\Cov{{\sf Cov}}%
\def\Var{{\sf Var}}%
\def\ie{i.e.}%
\def\eg{e.g.}%
\newcommand{\DF}{\,\stackrel{\sf def}{=}\,}%
\newcommand{\one}{\hbox{\rm 1\kern-.27em I}}%
\def\BbbR{\ensuremath{\mathbb{R}}}%

\def\bfC{\ensuremath{\mathbf{C}}}%
\def\bfP{\ensuremath{\mathbf{P}}}%

\def\calA{\ensuremath{\mathcal{A}}}%
\def\calB{\ensuremath{\mathcal{B}}}%
\def\calH{\ensuremath{\mathcal{H}}}%
\def\calI{\ensuremath{\mathcal{I}}}%
\def\calJ{\ensuremath{\mathcal{J}}}%
\def\calN{\ensuremath{\mathcal{N}}}%

\def\sfE{\ensuremath{\mathsf{E}}}%
\def\sfL{\ensuremath{\mathsf{L}}}%
\def\sfP{\ensuremath{\mathsf{P}}}%
\def\sfQ{\ensuremath{\mathsf{Q}}}%
\def\sfR{\ensuremath{\mathsf{R}}}%
\def\sfT{\ensuremath{\mathsf{T}}}%

\def\sfp{\ensuremath{\mathsf{p}}}%
\def\sfq{\ensuremath{\mathsf{q}}}%
\def\eps{\varepsilon}%
\def\vfi{\varphi}%
\newcommand{\bs}[1]{\boldsymbol{#1}}%
\def\bvfi{\boldsymbol{\varphi}}%
\def\xiL{\xi_{\sfL}}%
\def\xiR{\xi_{\sfR}}%
\newcommand{\calINxi}[1]{\calI_{N,#1}^{\,\boldsymbol{\xi}}}%
\newcommand{\bfPNxi}[1]{\sfP_{N,#1}^{\,\boldsymbol{\xi}}}%
\newcommand{\calBnu}{\calB^{\kern1pt\nu}_u}%

\newcommand{\theN}{\theta_N}%
\journal{Stochastic Processes and their Applications}

\begin{document}

\begin{frontmatter}

\title{Some Rigorous Results on Semiflexible Polymers\\I. Free and confined polymers}

\author[Durham]{O.~Hryniv}
\ead{Ostap.Hryniv@durham.ac.uk}

\author[Geneva]{Y.~Velenik\corref{cor}\fnref{fn1}}
\fntext[fn1]{Supported in part by the Swiss National Science Foundation}
\ead{Yvan.Velenik@unige.ch}

\cortext[cor]{Corresponding author}

\address[Durham]{Department of Mathematical Sciences, University of Durham,
Science Laboratories, South Rd, Durham DH1~3LE, UK}

\address[Geneva]{Section de Math\'ematiques, Universit\'e de Gen\`eve, 2-4, rue du Li\`evre, 1211 Geneva 4, Switzerland}

\begin{abstract}
We introduce a class of models of semiflexible polymers. The latter are characterized by a strong rigidity, the correlation length associated to the gradient-gradient correlations, called the persistence length, being of the same order as the polymer length.

We determine the macroscopic scaling limit, from which we deduce bounds on the free energy of a polymer confined inside a narrow tube.
\end{abstract}


\end{frontmatter}

\section{Introduction and results}
The purpose of the present work is to introduce and study a family of effective models of semiflexible polymers. The latter are polymers endowed with two characteristic properties: 1) Their thermal fluctuations are governed by their \emph{bending energy}, rather than their tension; in other words, they try to minimize curvature rather than length. 2) Their \emph{persistence length}, which can be roughly defined as the correlation length associated to the directional correlations between tangent vectors to the polymer, is of a size comparable to that of the polymer. Such semiflexible polymers play a crucial role in nature. In particular, the biological function of many biopolymers (such as DNA, filamentous actin or microtubules) relies on their semiflexibility, the latter providing considerable mechanical rigidity.

\subsection{The model}\label{sec:model}
The model most often used in the physics literature is the so-called \emph{wormlike chain}. In this model, the polymer is described by a smooth path in $\mathbb{R}^2$ (higher dimensions are of course possible, but we'll stick to this case in this paper), of fixed length $1$, which we'll write $\mathbf{r}(s)$ with $s$ being the parametrization such that the tangent vector $\mathbf{t}(s)=\mathrm{d}\mathbf{r}/\mathrm{d}s$ satisfies $\|\mathbf{t}(s)\|=1$ for all $s$. The energy functional associated to the path is then given by
$$
\int_0^1 \Phi\Bigl(\Bigl\|\frac{\mathrm{d}\mathbf{t}(s)}{\mathrm{d}s}\Bigr\|\Bigr)\,\mathrm{d}s,
$$
where $\Phi$ is usually taken as $\Phi(x)=\kappa x^2$, the parameter $\kappa$ setting the rigidity of the polymer. When the polymer makes only small deviations from the horizontal axis, an effective representation of the polymer as the graph of a function $f:\mathbb{R}\to\mathbb{R}$ becomes possible, associating to a given polymer configuration $f$ the energy
$$
\int_0^c\Phi\bl(f''(x)\br)\,dx,
$$
where $c$ is the macroscopic length of the polymer~\cite{twB93}. 

The main aim of the  present paper is to study discrete approximations to such models. Namely, we consider lattice configurations $\bvfi$ in the ensemble
$$
\calI_N^\xi \DF\Bl\{\bvfi = (\vfi_0=0,\vfi_1=\xi,\dots,\vfi_{N},\vfi_{N+1})\in\mathbb{Z}^{N+2}\Br\},
$$
equipped with the probability measure
\begin{equation}
\label{eq:free.measure}
\sfP_N^\xi(\bvfi)\propto
\exp\Bl\{-\calH_N\bl(\bvfi\br)\Br\}\,,
\end{equation}
where the Hamiltonian $\calH_N(\bvfi)$ is defined by
\begin{equation}
\label{eq:Ham}
\calH_N(\bvfi)\DF\eps\sum_{j=1}^{N}\Phi\bl(\eps^{-1}\Delta\vfi_j\br)\,.
\end{equation}
Here we assume that the parameters $N$ and $\eps$ satisfy $N\eps\to c$ as $N\to\infty$, where $c>0$ denotes the macroscopic length of the polymer, and use the standard notation for the lattice difference operators
\begin{equation}
\Eqlab{grad.def}
(\nabla\vfi)_k\DF\vfi_k-\vfi_{k-1}\,,\qquad
\Delta\vfi_k\DF(\nabla\bl(\nabla\vfi))_{k+1}
\equiv\vfi_{k+1}-2\vfi_k+\vfi_{k-1}\,.
\end{equation}
Similarly, if the gradient condition on the right end of the polymer becomes improtant, we consider the ensemble $\calINxi d$ (with $\boldsymbol{\xi}=\{\xiL,\xiR\}$ and $d=d_{N+1}$) of configurations with fixed endpoints and fixed gradients at both extremities,
\begin{align}
\label{eq:calInxio.def}
\calINxi d&\DF\Bl\{\bvfi = (\vfi_0=0,\dots,\vfi_{N+1})\,:\\
&\hspace{2cm}\vfi_1=\xiL,\vfi_{N}=\xiR+d_{N+1},\vfi_{N+1}=d_{N+1}\Br\}\nonumber\\
&\,\equiv\Bl\{\bvfi = (\vfi_0=0,\dots,\vfi_{N+1})\,:\nonumber\\
&\hspace{2cm}\nabla\vfi_1=\xiL,\nabla\vfi_{N+1}=-\xiR,\vfi_{N+1}=d_{N+1}\Br\}\nonumber
\end{align}
equipped with the probability measure
\begin{equation}
\label{eq:pinned.measure}
\bfPNxi d(\bvfi)\DF \sfP_N^{\xiL}(\bvfi \,|\, \bvfi\in\calINxi d)\,.
\end{equation}
Our aim is to describe the typical behaviour of the trajectories
$\bvfi\in\calINxi d$ under the measure $\bfPNxi d(\cdot)$ from
\Eqref{pinned.measure} with $N\eps\approx c$ and $\eps\ll 1$. Despite the Hamiltonian \Eqref{Ham} might look unusual, our results in App.~\ref{sec:heuristic} show that this choice corresponds to the semiflexibility regime, when the persistence length and the polymer length are of the same order. In addition, our results in Sect.~\ref{sec:scaling} below (see, eg., Remark~\ref{rem:theta-fclt} and Remark~\ref{rem:Gauss}) show that the Hamiltonian \Eqref{Ham} with $\Phi(x) \sim x^2$ as $x\sim0$ is essentially the only sensible choice from the physical point of view.

\bigskip
Of course, this model shares the limitations of the macroscopic effective model it approximates: It forbids backtracks of the polymer, and the gradients of $\varphi$ have to remain close to zero. For the questions we have in mind, these approximations will be harmless.

Let us nevertheless mention that it is also possible to discretize directly the wormlike chain, thus obtaining discrete models of semiflexible polymers that are better suited to discuss other properties (for example the effect of an external force pulling the polymer, which in general results in a non-trivial macroscopic profile). In particular, there is a natural discrete variant of the wormlike chain, in which the polymer is modelled as a chain of hard rods of unit-length, with an energy penalizing changes of orientation. This model might also be amenable to a mathematical analysis, although this would surely generate additional technicalities.

To our knowledge, the mathematical analysis of models of the form introduced above is still quite limited. The works closest to ours are~\cite{CaDe-I, CaDe-II}, in which the effect of an external pinning potential, similar to the problem we analyse in Part II~\cite{HrVe_PartII}, is done in the case $\eps=1$, as $N\to\infty$, both with and without a positivity constraint. In particular, it is shown that such models display a very different critical behaviour from that for interfaces or polymers with tension. Notice however, that setting $\eps=1$ and taking $N$ to infinity implies that the described polymer is not semiflexible anymore (its persistence length being of the order of the lattice spacing, while its length becomes infinite).

Other relevant works deal with the case of membranes, a natural higher-dimensional analogue of the one-dimensional polymer considered here. These models have also important applications, as they can be used to describe, \eg, cell membranes. However, their rigorous analysis is quite involved, and up to now only the case of objects of internal dimension at least $4$ have been successfully studied; see \cite{K08} and references therein.

We finally observe that to simplify our exposition we only discuss discrete height models. Similar results can also be obtained by analogous methods for continuous height models; of course, there one has to understand the RHS of \eqref{eq:free.measure} and of similar expressions as the densities w.r.t.\ the Lebesgue measure. The key ingredient of our analysis--the local limit theorem--remains the same. We refer the interested reader to the classical monograph \cite{P75}, Chap.~VII of which deals with LLT's both in discrete and continuous setting.

\section{Scaling properties of semiflexible polymers}\label{sec:scaling}
\subsection{Reduction to the RW case}

The problem above can be reduced to a problem about random walks. To
do this, consider the process
\begin{equation}
\label{eq:xi-walk.def}
\xi_k\DF(\nabla\vfi)_k\,,\qquad\xi_1\equiv\xiL\,,
\end{equation}
and observe that, w.r.t.\ the distribution from \Eqref{free.measure}-\Eqref{Ham}, its rescaled
increments
\[
\eta_k\DF \eps^{-1}\Delta\vfi_k\equiv \eps^{-1}\nabla\xi_{k+1}=\eps^{-1}(\xi_{k+1}-\xi_k)
\] 
are i.i.d.\ random variables. We need to study the large-$N$ behaviour of such random walks conditioned on the event $\xi_N\equiv-\xiR$. Since in view of \Eqref{xi-walk.def}
\begin{equation}
\label{eq:xi-eta.relation}
\xi_m=\xi_1+\sum_{j=2}^m\nabla\xi_j\equiv\xi_1+\eps\sum_{j=1}^{m-1}\eta_j\,,
\end{equation}
the ``gradient'' boundary condition in \Eqref{calInxio.def} reads
\begin{equation}
\label{eq:grad.bc}
\eps\sum_{j=1}^{N}\eta_j=\xi_{N+1}-\xi_1\equiv\nabla\vfi_{N+1}-\nabla\vfi_1\,.
\end{equation}
Now, using the relation \Eqref{xi-eta.relation}, we get
\begin{equation}
\label{eq:vfi-eta.relation}
\vfi_k=\vfi_1+\sum_{m=2}^k\xi_m
=k\xi_1+\eps\sum_{j=1}^{k-1}\bl(k-j\br)\eta_j
\end{equation}
and rewrite the ``right-end'' boundary condition from
\Eqref{calInxio.def} as 
\begin{equation}
\label{eq:pinning.bc}
\eps\sum_{j=1}^{N}\bl(N+1-j\br)\eta_j=\vfi_{N+1}-(N+1)\xi_1\equiv\vfi_{N+1}-(N+1)\nabla\vfi_1\,.
\end{equation}

Clearly, the problem above now reads as the problem of describing
conditional distribution of a RW with i.i.d.\ steps $\eta$, subject
to constraints \Eqref{grad.bc} and~\Eqref{pinning.bc}.
Questions of this type are well understood, see \cite{rDoH96} for a
recent treatment of a similar model, so we can simply state the corresponding answers and discuss the necessary modifications in the proofs.

\subsection{Functional CLT}\label{sec:FCLT}

We now turn to the analysis of the fluctuations of the process. In view of the applications we have in mind, and the intrinsic limitations of this model, we shall restrict attention to ``macroscopically bounded'' boundary conditions (see Theorem~\ref{thm:theta-fclt} below).

\bigskip
For $k=1,\dots,N$, we consider
\begin{equation}
\label{eq:XYk.def}
X_k\DF\sum_{j=1}^k\eta_j\,,\qquad 
Y_k\DF\frac1{N+1}\sum_{j=1}^k\bl((k+1)-j\br)\eta_j\,.
\end{equation}
Our basic assumption is (remember that $N=c/\epsilon$) that the first two moments of $\eta$ satisfy
\footnote{
Here and below we use $\sfE$, $\Var$ and $\Cov$ to denote the expectation, the variance and the covariance of various random variables expressed in terms of the $\eta$-variables. We stress that with fixed value of the first gradient $\xi$ and fixed law of the i.i.d.\ increments $\eta$, the probability measure $\sfP_N^\xi(\bvfi)$ becomes uniquely defined.
}
\begin{equation}
\label{eq:eta-varaibles-variance}
\sfE\eta=0\,,\qquad \sfE \eta^2=\sigma_N^2,\qquad\lim_{N\to\infty} N \sigma_N^2 =\infty.
\end{equation}
It implies, for all $m=1,2,\dots,N$,
\begin{equation}
\label{eq:XY-moments}
\begin{gathered}
 \sfE X_m=\sfE Y_m=0\,,\quad \Var X_m=m\sigma_N^2\,,
\\[1ex]
\sfE \bl(X_mY_m\br)=\frac{m(m+1)}{2(N+1)}\sigma_N^2\,,\quad \Var Y_m=\frac{m(m+1)(2m+1)}{6(N+1)^2}\sigma_N^2\,;
\end{gathered}
\end{equation}
in particular, the vector $Z_N\DF(X_N,Y_N)$ has zero mean and the covariance matrix
\begin{equation}
\label{eq:Z-Cov}
\Cov\bl(Z_N\br)=\begin{pmatrix} N\sigma_N^2 & \frac{N}2\sigma_N^2 \\[1ex] \frac{N}2\sigma_N^2 & \frac{N(2N+1)}{6(N+1)}\sigma_N^2\end{pmatrix}\,.
\end{equation}

We are going to study the asymptotics of the conditional process $\theN(t)$, $t\in[0,1]$, related to the one-point projections
\[
 \bl(Y_k\mid X_N=a_N,Y_N=b_N\br)
\]
with $a_N$, $b_N$ chosen in such a way that the probability of the condition
\[
 \sfP\bl(X_N=a_N,Y_N=b_N\br)
\]
remains positive for all $N$ large enough and, for some finite $K>0$,
\[
 \limsup_{N\to\infty}\frac{|a_N|+|b_N|}{\sigma_N\sqrt{N}}<K\,.
\]
More precisely, for $t\in[0,1]$ let
\begin{equation}
\label{eq:Nt-def}
N_t\DF\bl[Nt\br]\,,
\end{equation}
and define the continuous process $\theN(t)$ via
\begin{equation}
\label{eq:thetaN-process}
\theN\Bl(\frac mN\Br)=\frac1{\sigma_N\sqrt{N}}\bl(Y_m\mid X_N=a_N,Y_N=b_N\br)
\end{equation} 
at the points $t=m/N\in[0,1]$ with subsequent linear interpolation for other values of $t\in[0,1]$. Our main result reads as follows:

\begin{theorem}\label{thm:theta-fclt}\sl
Let the independent random variables $\eta$ have common distribution with variance $\sigma_N^2$ satisfying $N\sigma_{N}^2\to\infty$ as $N\to\infty$. If
\[
 \lim_{N\to\infty}\frac{a_N}{\sigma_N\sqrt{N}}=a\,,\qquad \lim_{N\to\infty}\frac{b_N}{\sigma_N\sqrt{N}}=b\,,
\]
then the distribution of $\theN(t)$ converges weakly in $\bfC[0,1]$ to that of a Gaussian process $\theta(t)$, $t\in[0,1]$, such that for all $s$, $t$ with $0\le s\le t\le1$,
\begin{equation}
\label{eq:theta-moments}
\begin{gathered}
\sfE\theta(t)=t^2(t-1)a+t^2(3-2t)b\,,\\[1ex]
 \Cov\bl(\theta(s),\theta(t)\br)=\frac{s^2(1-t)^2}6\Bl(2t(1-s)+t-s\Br)\,.
\end{gathered}
\end{equation}
In particular, for $a=b=0$ we get
\[
 \theta(t)\sim \calN\bl(0,\tfrac13 t^3(1-t)^3\br)\,,\quad t\in[0,1]\,.
\]
\end{theorem}

\begin{remark}\label{rem:theta-fclt}\rm
\begin{enumerate}
\item In a sense, the main message of the above result is that under sufficiently mild assumptions (i.i.d.\ increments $\eta$ with variance $\sigma^2_N$ satisfying $N\sigma_{N}^2\to\infty$ as $N\to\infty$) the only physically relevant potentials $\Phi$ for the model at hand are convex potentials of the Gaussian type, $\Phi(x)\sim\kappa x^2/2$ as $x\to0$.
Indeed, for every model satisfying Theorem~\ref{thm:theta-fclt}, there exists a {\it mesoscopic\/} scale $\delta=\delta_N\to0$ such that $N_\delta=c/\delta\to\infty$ will still satisfy the condition $N_\delta\sigma^2_N\to\infty$. As a result, it is possible to discretize our macroscopic polymer so that its behaviour on the scale $\delta$ is approximately Gaussian. Consequently, among various a priori legitimate choices $\Phi(x)\sim|x|^\alpha$, the Gaussian case $\Phi(x)\sim|x|^2$, popular in physics literature, seems most natural for the problems we discuss here.

\item The reader might wish to interpret the limiting process $\theta(t)$ as the ``bridge of the integral of a Brownian bridge''. Indeed, with
\[
 \bs{Y}=\bl(Y_0,Y_1,\dots,Y_N)\,,\qquad Y_k\equiv\frac1{N+1}\sum_{j=1}^kX_j
\]
and $\bs{X}=(X_0,X_1,\dots,X_N)$ satisfying the invariance principle, the scaling limit of $\bs{Y}$ becomes the integral of the scaling limit of $\bs{X}$, ie., a Brownian motion; see the proof of Theorem~\ref{thm:theta-fdd} below. Of course, a similar interpretation holds for other results in this section. We shall leave such observations as an exercise for a motivated reader.

Also, one might wish to notice that the function $m(t)=\sfE\theta(t)$ satisfies
\[
 m(0)=m'(0)=0\qquad\text{ and }\qquad m(1)=b\,,\quad m'(1)=a\,,
\]
which is not surprising since our choice of the exponents $\gamma$ and $\delta$ in  \Eqref{scaling-relation} guarantees that the limiting process $\theta(t)$ shares common gradient restrictions with all its discretizations.
\end{enumerate}
\end{remark}

Rewriting \Eqref{vfi-eta.relation} in the form
\[
\vfi_{N_t+1}\equiv(N_t+1)\xi_1+(N+1)\eps\,Y_{N_t}\,,
\]
we observe that the boundary conditions \Eqref{calInxio.def},
\begin{equation}
\label{eq:vfi-boundary-fclt}
 \vfi_0 = 0\,,\quad \vfi_1 = \xi_L\,,\quad \vfi_{N+1} = d_{N+1}\,,\quad \vfi_N = d_{N+1} + \xi_R\,,
\end{equation}
become
\begin{equation}
\label{eq:XY-boundary-fclt}
X_N = - ( \xi_L + \xi_R ) / \eps\,,\quad   Y_N = (d_{N+1}/(N+1) - \xi_L )/ \eps
\end{equation}
and we can rewrite the theorem above in terms of the ``profile process'' $\vfi$:
\begin{corollary}\sl
\label{cor:CLT}
Conditioned on \Eqref{vfi-boundary-fclt} with the property
\[
- ( \xi_L + \xi_R ) / (\eps\sigma_N\sqrt{N})\to a\,,\qquad (d_{N+1}/(N+1) - \xi_L )/(\eps\sigma_N\sqrt{N})\to b
\]
as $N\to\infty$ (recall that $N\eps\approx c$) the distribution of the process
\[
 \bl(\vfi_{N_t+1}-(N_t+1)\xi_L\br)/\bl(\sigma_N\sqrt{N}\,(N+1)\eps\br)\,,\qquad t\in[0,1]\,,
\]
w.r.t.\ to the probability measure $\bfPNxi d(\cdot)$ from \Eqref{pinned.measure} converges weakly in $\bfC[0,1]$ to the limiting Gaussian distribution with parameters \Eqref{theta-moments}; in particular, its one-dimensional distributions approach
\[
\calN\Bl(t^2(t-1)a+t^2(3-2t)b,\tfrac13 t^3(1-t)^3\Br)
\]
as $N\to\infty$.
\end{corollary}

\begin{remark}\label{rem:Gauss}\rm
In the most popular case considered in the physical literature, namely the Gaussian case $\Phi(x) \sim x^2$ with the Hamiltonian (cf.~\Eqref{Ham})
\[
\calH_N(\bvfi)\equiv\frac\kappa2\sum_{j=1}^{N}\frac{\bl(\Delta\vfi_j\br)^2}{\eps}\,,
\]
the random variables $\eta_j\equiv\eps^{-1}\Delta\vfi_j$ have variance $\sigma_N^2=O(\eps^{-1})=O(N/c)$, so that by the corollary above the fluctuations of the polymer are of order
\[
\sigma_N\sqrt{N}\,(N+1)\eps\sim\sqrt{c}N = c^{3/2} \eps^{-1}\,,
\]
ie., they live on the macroscopic scale.
\end{remark}

The rest of this section is devoted to the proof of Theorem~\ref{thm:theta-fclt}. We first derive convergence of finite-dimensional distributions of the process $\theN(t)$ (see Theorem~\ref{thm:theta-fdd} below) and then establish tightness of the probability distributions of $\theN(t)$ in $\bfC[0,1]$.

\medskip
We turn now to the proof of Theorem~\ref{thm:theta-fclt}. Let $\chi(u)$, $u\in\BbbR$, denote the characteristic function of $\eta$,
\[
 \chi(u)\DF\sfE\exp\bl\{iu\eta\br\}\,;
\]
by the moment assumption above we have:
\begin{equation}
\label{eq:chi-smallU-expansion}
\chi(u/\sigma_N)=1-\frac{u^2}2+o(u^2)\qquad\text{ as $u\to0$.}
\end{equation} 
Fix an integer $k\ge0$ and a collection of real numbers $t_j$ satisfying
\begin{equation}
\label{eq:tj-def}
0\equiv t_0<t_1<t_2<\dots<t_k<t_{k+1}\equiv1\,.
\end{equation} 
Our first goal is to prove the central limit theorem for the random vector
\begin{equation}
\label{eq:ZNk-def}
Z_N^k\DF\frac{1}{\sigma_N\sqrt{N}}\bl(X_N,Y_{N_{t_1}},Y_{N_{t_2}},\dots,Y_{N_{t_k}},Y_{N_{t_{k+1}}}\br)\,.
\end{equation} 
To this end, observe that the corresponding characteristic function reads
\[
 \overline{\chi}_N^k\bl(u_0,u_1,\dots,u_{k+1}\br)
=\sfE\exp\Bl\{\frac{i}{\sigma_N\sqrt{N}}\Bl(u_0X_N+\sum_{l=1}^{k+1}u_lY_{N_{t_l}}\Br)\Br\}\,.
\]
It is convenient to denote
\begin{equation}
\label{eq:uk-jx-def}
\begin{gathered}
 u_N^k(j)\DF u_0+\sum_{l=1}^{k+1}\frac{u_l}{N+1}(N_{t_l}+1-j)^+\,,\quad 1\le j\le N\,,
\\
u^k(x)\DF u_0+\sum_{l=1}^{k+1}u_l(t_l-x)^+\,,\quad 0\le x\le1\,.
\end{gathered}
\end{equation} 
Then with $\bs{\bar u}=(u_0,u_1,\dots,u_{k+1})^\sfT\in\BbbR^{k+2}$ we rewrite
\begin{equation}
\label{eq:barchiNk-product}
\overline{\chi}_N^k(\bs{\bar u})=\prod_{j=1}^N\chi\Bl(u_N^k(j)/(\sigma_N\sqrt{N})\Br)
\end{equation} 
so that, in view of the asymptotic relation \Eqref{chi-smallU-expansion} and the limiting assumption $\sigma_N^2 N\to\infty$ (as $N\to\infty$), we get, uniformly in $\bs{\bar u}$ from compact sets in $\BbbR^{k+2}$,
\begin{equation}
\label{eq:log-barchiNk-expansion}
\log\overline{\chi}_N^k(\bs{\bar u})
=-\frac{1}{2N}\sum_{j=1}^N\bl[u_N^k(j)\br]^2+o(1)
=-\frac{1}2\int_0^1\bl[u^k(x)\br]^2\,dx+o(1)\,.
\end{equation}
By a routine (but straightforward!) induction one deduces the following result:

\begin{lemma}\label{lem:Qmatrix}\sl
 For $s$, $t$ with $0\le s\le t\le1$ denote
\[
 f(t)\equiv\frac{t^2}2\,,\qquad g(s,t)\equiv\frac{s^2}6(3t-s)\,.
\]
Then for every integer $k\ge0$ the quadratic form
\[
 \int_0^1\bl[u^k(x)\br]^2\,dx=\sum_{l_1,l_2=0}^{k+1}\sfq_{l_1,l_2}\,u_{l_1}u_{l_2}
\]
has the matrix
\begin{equation}
\label{eq:Q-matrix}
\sfQ^k=\bl[\sfq_{l_1,l_2}\br]_{l_1,l_2=0}^{k+1}
\end{equation} 
with the entries
\begin{equation}
\label{eq:qlm-entries}
\begin{gathered}
\sfq_{00}=1\,,\qquad \sfq_{0l}=\sfq_{l0}=f(t_l)\,,\quad 1\le l\le k+1\,,
\\[1ex]
\sfq_{l_1,l_2}=\sfq_{l_2,l_1}=g(t_{l_1},t_{l_2})\,,\quad 1\le l_1\le l_2\le k+1\,.
\end{gathered}
\end{equation} 
\end{lemma}

The above lemma together with \Eqref{log-barchiNk-expansion} imply the central limit result:

\begin{theorem}[Central Limit Theorem]\sl
\label{thm:CLT}
For every fixed $k\ge0$ the distribution of the random vector $Z_N^k$ from \Eqref{ZNk-def} converges as $N\to\infty$ to the Gaussian distribution with zero mean and the covariance matrix $\sfQ^k$ defined in \Eqref{Q-matrix}--\Eqref{qlm-entries}. The convergence of the corresponding characteristic functions \Eqref{log-barchiNk-expansion} is uniform on compact subsets of~$\BbbR^{k+2}$.
\end{theorem}

\begin{remark}\rm
 According to \Eqref{ZNk-def}, this theorem implies that the fluctuations of $X_N$ are of order $\sigma_N\sqrt{N}$. Combining this with \Eqref{grad.bc} and \Eqref{XYk.def}, we see that the end-to-end gradient fluctuations $\nabla\vfi_{N+1}-\nabla\vfi_1$ of the polymer $\bvfi$ are of order $\eps\sigma_N\sqrt{N}$. In the natural Gaussian scaling $\sigma_N^2=O(\eps^{-1})$ of Remark~\ref{rem:Gauss} this implies that
\[
 \Var\bl(\nabla\vfi_{N+1}-\nabla\vfi_1\br)=O(\eps^2\sigma_N^2N)=O(c)\,.
\]
In other words, the persistence length and the polymer length in our model are of the same order.
\end{remark}

Our next goal is to establish the local version of the above theorem. For
\[
 \bs{\bar x}=\bl(x_0,x_1,\dots,x_k,x_{k+1}\br)^{\sfT}\in\BbbR^{k+2}
\]
let $\sfp^k_{\sfQ}(\bs{\bar x})$ denote the probability density
\[
 \sfp^k_{\sfQ}(\bs{\bar x})=\bl(2\pi\br)^{-(k+2)/2}|\det\sfQ|^{-1/2}
\exp\Bl\{-\frac1{2}\bl(\sfQ^{-1}\bs{\bar x},\bs{\bar x}\br)\Br\}
\]
of the limiting Gaussian distribution with the characteristic function
\[
\overline{\chi}_{\sfQ}^k(\bs{\bar u})
=\exp\Bl\{-\tfrac12\bl(\sfQ\bs{\bar u},\bs{\bar u}\br)\Br\}\,.
\]
\begin{theorem}[Local CLT]\label{thm:FCLT}\sl
Let a sequence of vectors
\[
\bs{\bar x}^{(N)}
=\bl(x_0^{(N)},x_1^{(N)},\dots,x_k^{(N)},x_{k+1}^{(N)}\br)^{\sfT}\in\BbbR^{k+2}
\]
be such that $\bs{\bar x}^{(N)}\to\bs{\bar x}\in\BbbR^{k+2}$ as $N\to\infty$ and the probability $\sfP\bl(Z_N^k=\bs{\bar x}^{(N)}\br)$
be positive for all $N$ large enough. Then as $N\to\infty$ we have
\[
 \sigma_N^{k+2}N^{(3k+4)/2}\,\sfP\bl(Z_N^k=\bs{\bar x}^{(N)}\br)=\sfp^k_{\sfQ}(\bs{\bar x})+o(1)
\]
with the remainder $o(1)$ vanishing asymptotically, as $N\to\infty$, uniformly in $\bs{\bar x}$ on compact subsets of~$\BbbR^{k+2}$.
\end{theorem}

\begin{proof}
The claim of the theorem follows from standard considerations provided the off-line  property is established (for a recent exposition, see, \eg, \cite[Thm~4.2]{rDoH96}); it thus remains to verify the latter.

By the assumption on the distribution of $\eta$, we have, for all $\zeta>0$ small enough
\begin{equation}
\label{eq:chi-off-line}
\sup_{\zeta\le|u/\sigma_N|\le T}\bl|\chi(u/\sigma_N)\br|=r_\zeta\in(0,1)\,,
\end{equation}
where $T=\pi/d$ for lattice distributions of period $d>0$. In view of the factorization \Eqref{barchiNk-product}, the off-line property shall follow once we show that for some small enough $\zeta=\zeta_k>0$ sufficiently many values $u_N^k(j)$ satisfy the condition (recall \Eqref{uk-jx-def})
\[
 \bl|u_N^k(j)\br|\ge\zeta_k \sigma_N\,,
\]
uniformly in $N$ large enough. However, by the very definition \Eqref{uk-jx-def}, the sequence $u_N^k(j)$, $j=1,2,\dots,N$ is a piecewise linear sequence of real numbers interpolating the values
\[
 u_N^k(N_{t_0})\,,\quad u_N^k(N_{t_1})\,,\quad\dots,\quad u_N^k(N_{t_{k+1}})
\]
and having increments (recall \Eqref{Nt-def})
\[
 u_N^k(j)-u_N^k(j+1)=\frac1{N+1}\sum_{l=1}^{k+1}u_l\one_{j<N_{t_l}}\,.
\]
By \cite[Lemma~4.4]{rDoH96} it is enough to show that for $\zeta>0$ as in \Eqref{chi-off-line} one has (recall \Eqref{Nt-def}, \Eqref{uk-jx-def})
\[
 \max_{l=0,\dots,k+1}\bl|u_N^k(N_{t_l})\br|>2\zeta\sigma_N\,,
\]
as then the rest of the proof of Theorem~\ref{thm:FCLT} would be analogous to that of \cite[Thm~4.2]{rDoH96}.
\end{proof}

We prove the remaining condition by verifying the following claim.

\begin{lemma}\sl
 For a fixed collection $t_1$, \dots, $t_k$ as in \Eqref{tj-def}, let
\[
 \Delta=\min_{m\ge0}\bl(t_{m+1}-t_m\br)>0\,.
\]
Then for every $\bs{\bar u}\in\BbbR^{k+2}$ such that 
\[
 \bl\|\bs{\bar u}\br\|_2=\Bl(\sum_{l=0}^{k+1}(u_l)^2\Br)^{1/2}>\frac{2\eta\sigma_N}\Delta\sqrt{16k+5}
\]
and all $N$ large enough at least one of the following inequalities holds:
\begin{equation}
\label{eq:unk-increment-bound}
\bl|u_N^k(N)\br|>2\zeta\sigma_N\,,\qquad \bl|u_N^k(N_{t_l})-u_N^k(N_{t_{l+1}})\br|>4\zeta\sigma_N \quad l=0,\dots,k\,.
\end{equation}
\end{lemma}

\begin{proof}
We argue by contradiction and start by assuming that none of the inequalities \Eqref{unk-increment-bound} holds. Since 
\[
 u_N^k(N)=u_0+u_{k+1}/(N+1)
\]
and
\[
 u_N^k(N_{t_m})-u_N^k(N_{t_{m+1}})=(t_{m+1}-t_m)\sum_{l>m}u_l+O(N^{-1})
\]
we deduce that
\[
\bl|u_{k+1}\br|\le\frac{4\zeta\sigma_N}\Delta\,,\quad
\bl|u_k\br|\le\frac{8\zeta\sigma_N}\Delta\,,\quad \dots\,,\quad
\bl|u_1\br|\le\frac{8\zeta\sigma_N}\Delta
\]
and therefore that
\[
\bl\|\bs{\bar u}\br\|_2^2\le4\zeta^2\sigma_N^2+\Bl(\frac{4\zeta\sigma_N}\Delta\Br)^2\,(4k+1)
\le\Bl(\frac{2\zeta\sigma_N}\Delta\Br)^2\,(16k+5)\,.
\]
\end{proof}

We now deduce convergence of finite-dimensional distributions of the process $\theN(\cdot)$ from \Eqref{thetaN-process}:

\begin{theorem}\label{thm:theta-fdd}\sl
 Let real sequences $a_N$, $b_N$ be such that 
\[
 \lim_{N\to\infty}\frac{a_N}{\sigma_N\sqrt{N}}=a\,,\qquad \lim_{N\to\infty}\frac{b_N}{\sigma_N\sqrt{N}}=b
\]
and the probability $\sfP(X_N=a_N,Y_N=b_N)$ be positive for all $N$ large enough. Then for every $k\ge1$ the $k$-dimensional distributions of the process $\theN(\cdot)$ converge to those of a Gaussian process $\theta(\cdot)$, whose parameters are
\begin{equation}
\label{eq:theta-process-parameters}
\begin{gathered}
 \sfE\theta(t)=t^2(t-1)a+t^2(3-2t)b\,,\quad t\in[0,1]\,,
\\
\Cov\bl(\theta(s),\theta(t)\br)=\frac{s^2(1-t)^2}6\bl[2t(1-s)+t-s\br]\,,\quad 0\le s\le t\le1\,.
\end{gathered}
\end{equation}
\end{theorem}

\begin{proof}
As the convergence result follows directly from the local limit theorem, we shall only derive the parameters \Eqref{theta-process-parameters} of the limiting process~$\theta(t)$.

To start, fix $0\le s\le t\le1$ and notice that the conditional distribution of
\[
 \frac1{\sigma_N\sqrt{N}}\Bl(Y_{N_s},Y_{N_t}\mid X_N=a_N,Y_N=b_N\Br)
\]
converges to that of 
\[
 \bl(\calJ(s),\calJ(t)\bigm|w_1=a,\calJ(1)=b\br)\,,
\]
where
\[
 \calJ(v)\DF\int_0^1(v-u)^+\,dw_u\equiv\int_0^v(v-u)\,dw_u\equiv\int_0^vw_u\,du
\]
and $w_s$, $s\in[0,1]$ is the standard Brownian motion (Wiener process). Using Lemma~\ref{lem:Qmatrix} and the classical property of conditional multivariate Gaussian distributions, we deduce that the mean of the limiting process equals
\[
\begin{split}
\sfE\bl(\calJ(t)\bigm|w_1=a,\calJ(1)=b\br)
&=\begin{pmatrix}f(t)&g(t,1)\end{pmatrix}
\begin{pmatrix}1&\frac12\\[1ex]\frac12&\frac13\end{pmatrix}^{-1}
\begin{pmatrix} a\\b\end{pmatrix}
\\[1ex]
&=t^2(t-1)a+t^2(3-2t)b
\end{split}
\]
and its covariance matrix is
\[
\begin{split}
\begin{pmatrix}h(s,s)&h(s,t)\\h(s,t)&h(t,t)\end{pmatrix}
&=\begin{pmatrix}g(s,s)&g(s,t)\\g(s,t)&g(t,t)\end{pmatrix}
\\
&-\begin{pmatrix}f(s)&g(s,1)\\f(t)&g(t,1)\end{pmatrix}
\begin{pmatrix}1&\frac12\\[1ex]\frac12&\frac13\end{pmatrix}^{-1}
\begin{pmatrix}f(s)&f(t)\\g(s,1)&g(t,1)\end{pmatrix}
\end{split}
\]
with
\[
 h(s,t)=\frac{s^2(1-t)^2}6\bl[2t(1-s)+t-s\br]\,,\quad 0\le s\le t\le1\,.
\]
\end{proof}

It remains to prove tightness of the sequence of probability distributions of the processes $\theN(\cdot)$ in the space $\bfC[0,1]$ of continuous functions on $[0,1]$. To this end it is sufficient (\cite[Thm~9.2.2]{iiGavS69}) to show that for some positive $C$ and $\gamma>1$ the inequality 
\begin{equation}
\label{eq:theN-tightness-bound}
\sfE\bl|\theN(t)-\theN(s)\br|^2\le C|t-s|^\gamma
\end{equation}
holds uniformly in $[s,t]\subseteq[0,1]$ and all $N$ large enough.\footnote{
Actually, our argument shows that here $\gamma=2$; this is not surprising, as the trajectories of the limiting process $\theta(\cdot)$ have continuous derivatives.
}
The key to \Eqref{theN-tightness-bound} is the following result whose proof shall be postponed till the end of the section.

\begin{lemma}\label{lem:increment-tightness}\sl
 Let real sequences $a_N$, $b_N$ be such that 
\[
 \lim_{N\to\infty}\frac{a_N}{\sigma_N\sqrt{N}}=a\,,\qquad \lim_{N\to\infty}\frac{b_N}{\sigma_N\sqrt{N}}=b
\]
and the probability $\sfP(X_N=a_N,Y_N=b_N)$ be positive for all $N$ large enough. There exists a positive constant $C_1$ such that the inequality
\[
 \sfE\bl(X_k^2\bigm|X_N=a_N,Y_N=b_N\br)\le C_1 \sigma_N^2 N
\]
holds uniformly in $k=1,2,\dots,N$.
\end{lemma}

The target condition \Eqref{theN-tightness-bound} is a straightforward corollary of the above lemma. Indeed, it follows from \Eqref{XYk.def}, the definition \Eqref{thetaN-process} and the lemma that for every $m=1,2,\dots,N$
\[
\begin{split}
\sfE\Bl|\theN\Bl(\frac mN\Br)-\theN\Bl(\frac{m-1}N\Br)\Br|^2
&=\frac1{\sigma_N^2 N(N+1)^2}\sfE\Bl(X_m^2\Bigm|X_N=a_N,Y_N=b_N\Br)
\\&\le\frac{C_1}{(N+1)^2}\,.
\end{split}
\]
Now, observing that for all $s$, $t$ with $0\le s\le t\le1$ we have
\[
\theN(t)-\theN(s)=\sum_{j=N_s+1}^{N_t+1}\alpha_N(j)
\Bl[\theN\Bl(\frac jN\Br)-\theN\Bl(\frac{j-1}N\Br)\Br]\,,
\]
where $\alpha_N(j)=1$ for all $j$ in the sum (with possible exception of the extreme values $j=N_s+1$ and $j=N_t+1$, for which $\alpha_N(j)\in[0,1]$), the Cauchy inequality gives
\[
\bl|\theN(t)-\theN(s)\br|^2\le\bl(N_t+1-N_s\br)\sum_{j=N_s+1}^{N_t+1}\bl|\alpha_N(j)\br|^2
\Bl|\theN\Bl(\frac jN\Br)-\theN\Bl(\frac{j-1}N\Br)\Br|^2
\]
and thus implies the target estimate \Eqref{theN-tightness-bound}:
\[
 \sfE\bl|\theN(t)-\theN(s)\br|^2\le C_1\frac{(N_t+1-N_s)^2}{(N+1)^2}\le C|t-s|^2\,,
\]
uniformly in $0\le s\le t\le1$ and all $N$ large enough.

We turn now to the proof of Lemma~\ref{lem:increment-tightness} and shall treat separately the two cases $k^2>N$ and $k^2\le N$.

\begin{proof}[of~Lemma~\ref{lem:increment-tightness}]

\noindent{\sf Case $k>\sqrt{N}$.} Let $k=k_N>\sqrt{N}$ and $k/N\to\kappa\in[0,1]$ as $N\to\infty$. Then for the vector
\[
 U^k\DF\Bl(\frac1{\sigma_N\sqrt{k}}\,X_k,\frac1{\sigma_N\sqrt{N}}\,X_N,\frac1{\sigma_N\sqrt{N}}\,Y_N\Br)
\]
the central limit theorem holds. Indeed, by a straightforward computation we deduce that the characteristic function of $U^k$ satisfies
\[
\begin{split}
\lim_{N\to\infty}&\log\sfE\exp\Bl\{i\Bl(\frac{v_0}{\sigma_N\sqrt{k}}\,X_k+\frac{v_1}{\sigma_N\sqrt{N}}\,X_N+\frac{v_2}{\sigma_N\sqrt{N}}\,Y_N\Br)\Br\}
\\
&=-\frac12\Bl(v_0^2+2v_0v_1\sqrt\kappa+v_1^2+2v_0v_2\sqrt\kappa\Bl(1-\frac\kappa2\Br)+v_1v_2+\frac{v_2^2}3\Br)\,.
\end{split}
\]
As the variance of the limiting conditional distribution is
\[
 \bl(1-4\kappa+6\kappa^2-3\kappa^3\br)\in[0,1]
\]
and its mean is bounded,%
\footnote{
being a linear combination of the constraints $a$ and $b$ (with $\kappa$-dependent coefficients);
}
we deduce that for some $C_2>0$ 
\[
\sfE\Bl(\Bl(\frac1{\sigma_N\sqrt{k}}\,X_k\Br)^2\Bigm|X_N=a_N,Y_N=b_N\Br)\le C_2
\]
uniformly in $k$ under consideration.

\noindent{\sf Case $k\le\sqrt{N}$.} Using arguments similar to those in \cite[pg.~257]{HrVe04}, we deduce that for all $j=1,2,\dots,N$
\[
 \sfE\Bl(\bl(\eta_j\br)^2\Bigm|X_N=a_N,Y_N=b_N\Br)\le C_3\sigma_N^2
\]
(in fact, as explained in \cite{HrVe04} for large $N$ the LHS is close to $\sfE\eta_j^2=\sigma^2_N$). As a result, the Cauchy inequality implies
\[
\begin{split}
\sfE\Bl(\bl(X_k\br)^2\Bigm|X_n=a_N,Y_N=b_N\Br)
&\le k^2\max\sfE\Bl(\bl(\eta_j\br)^2\Bigm|X_N=a_N,Y_N=b_N\Br)
\\
&\le C_3\sigma_N^2 k^2\le C_3 \sigma_N^2 N\,.
\end{split}
\]
The proof of Lemma~\ref{lem:increment-tightness} is finished.
\end{proof}

\subsection{Large deviation regime}

By combining the arguments above with the approach of \cite{rDoH96}, one can also describe the large deviation behaviour of semiflexible polymers. As such generalization is straightforward, we only mention some results.

Let $L_N(h)$ denote the log moment generating function of the step distribution (recall \Eqref{Ham}),
\begin{equation}
\label{eq:LMGF.def}
L_N(h)\DF\log\sfE\exp\bl\{h\eta\br\}\equiv
\log\frac{\int e^{-\frac cN\Phi(x)+hx}\,dx}{\int e^{-\frac cN\Phi(x)}\,dx}\,;
\end{equation} 
we shall assume that
\begin{equation}
\label{eq:LN-moments}
L_N'(0)=0\,,\qquad L_N''(0)=\sigma_N^2\in(0,+\infty)\,,
\end{equation}
that $L_N(\,\cdot\,)$ is finite in some (in general, $\eps$-dependent) neighbourhood of the origin, and that $L_N(\cdot)$ behaves properly under rescaling:
\begin{equation}
\label{eq:LN-scaling-limit}
L_N\Bl(\frac h{\sigma_N}\Br)\to L(h)\qquad\text{ as $N\to\infty$,}
\end{equation}
where $L(\cdot)$ is a strictly convex function in some $h$-neighbourhood of the origin. E.g., for the Gaussian case $\Phi(x)=\kappa x^2/2$ we obviously have
\[
 L_N(h)=\frac{N}{2c\kappa}\, h^2\qquad\text{ and }\qquad L(h)=\frac{1}{2}\, h^2\,.
\]

\subsubsection{Probability of the right-end boundary condition}
Let $X_m$, $Y_m$ be as defined in \Eqref{XYk.def}, 
\[
 X_m=\sum_{j=1}^m\eta_j\,,\qquad Y_m=\frac1{N+1}\sum_{j=1}^N\bl(m+1-j\br)^+\eta_j\,,
\]
and let $\sfP$ denote the probability distribution of the RW with steps $\eta_j$; we shall assume that the assumptions \Eqref{LMGF.def}--\Eqref{LN-scaling-limit} hold. Then the probability of the right-end boundary conditions
given the left-end ones (essentially of finishing a ``droplet'' at time $N$ with gradient $-\xiR$) is
\[
\begin{split}
&\sfP\bl(\bvfi\in\calINxi a\mid\vfi_0=0,\nabla\vfi_1=\xiL\br)
\\&\hphantom{\bfP_n\bl(\bvfi\in\calINxi a\mid}
\equiv\sfP\bl(\vfi_{N+1}=a(N+1),\nabla\vfi_{N+1}=-\xiR
\mid\vfi_0=0,\nabla\vfi_1=\xiL\br)
\end{split}
\]
and, in view of the relation
\begin{equation}
\label{eq:vfi-Y.relation}
\vfi_{[tN]+1}\equiv([tN]+1)\xi_1+(N+1)\eps\,Y_{[tN]}\,,
\end{equation}
clearly, coincides with the LD-type probability
\[
\sfP\Bl(X_N=-\eps^{-1}(\xiR+\xiL),Y_N=\eps^{-1}(a-\xiL)\Br)\,.
\]
Its limiting behaviour is well known (see, \eg, \cite[Theorem~4.2]{rDoH96}), so we just recall the corresponding result:

For real numbers $u$ and $v$, denote
\[
L_N(u,v)\DF\log\sfE\exp\Bl\{\frac u{\sigma_N}X_N+\frac v{\sigma_N}Y_N\Br\}\,;
\]
then, as $N\to\infty$, we have
\begin{equation}
\label{eq:LNoverN2int}
N^{-1}L_N(u,v)\to L_\infty(u,v)\DF\int_0^1L\bl(u+(1-x)v\br)\,dx
\end{equation}
with $L(\,\cdot\,)$ from \Eqref{LMGF.def}. The optimal tilts $u^*$, $v^*$ can be determined from the conditions (cf.~\cite[Eq.(2.26)]{rDoH96})
\begin{equation}
\label{eq:BC.tilts}
\left\{
\begin{aligned}
&\int_0^1L'(u+yv)\,dy=-\frac{\xiR+\xiL}c\,,
\\
&\int_0^1yL'(u+yv)\,dy=-\frac{\xiL-a}c\,,
\end{aligned}
\right.
\end{equation}
where we use the fact that $\vfi_N/N\to a$ as $N\to \infty$ in such a way that $N\eps\to c$, the
macroscopic length of the excursion under consideration.
Then the sharp LD asymptotics for the probability of interest, up
to a factor of $(1+o(1))$, is
\begin{equation}
\label{eq:LDP}
\frac1{2\pi N^2\sqrt{\|D(u^*,v^*)\|}}
\exp\Bl\{-N\Bl(-\frac{\xiR+\xiL}cu^*-\frac{\xiL-a}cv^*
-L_\infty(u^*,v^*)\Br)\Br\}\,,
\end{equation}
where $D(u,v)$ stands for the Hessian of $L_\infty$ as the function of $u$, $v$. Clearly, the expression in the exponential is just the convex dual $L^*_\infty$ of $L_\infty$ evaluated at the point with coordinates as in the RHS of \Eqref{BC.tilts}.

A straightforward computation in the Gaussian case $\Phi(x)=\kappa x^2/2$ based upon the correspondence \Eqref{vfi-boundary-fclt}--\Eqref{XY-boundary-fclt} as well as the moments \Eqref{XY-moments}--\Eqref{Z-Cov} gives the following {\bf exact} analogue of \Eqref{LDP} for $a=0$ and $N>1$:
\[
\label{eq:exact.Gauss}
\frac\kappa{2\pi N^2}\sqrt{\frac{12(N+1)}{N-1}}
\exp\Bl\{-\frac{(2N+1)\xiL^2-2(N+2)\xiL\xiR+(2N+1)\xiR^2}{c(N-1)/(N\kappa)}\Br\}\,.
\]

\subsubsection{Mean profile}

To catch the mean profile, fix a real $t$, $0<t<1$, and consider the
vector
\[
Z_N^t\DF\bl(X_N,Y_N,Y_{[tN]}\br)\,.
\]
Since according to the relation \Eqref{vfi-eta.relation} we have
the conditional distribution of $\vfi_{[tN]+1}$ given $(X_N,Y_N)$ can be
directly derived from the local limit theorem for the vector $Z_N^t$.

Mimicking \cite{rDoH96}, we introduce the log moment generating function
$L_N^t(u,v,w)$ of the vector $Z_N^t$,
\[
L_N^t(u,v,w)\DF\log\sfE\exp\Bl\{\frac u{\sigma_N}X_N+\frac v{\sigma_N}Y_N+\frac w{\sigma_N}Y_{[tN]}\Br\}\,,
\]
and observe that the conditional mean value of the last component
$Y_{[tN]}$ of $Z_N^t$ given the value of the first two is, up to a small correction,
as
\[
\sfE\bl(Y_{[tN]}\mid X_N=x_N,Y_N=y_N\br)\approx\sigma_N\frac\partial{\partial w}L_N^t(u^*,v^*,w)\biggm|_{w=0}
\]
with the optimal values $u^*$, $v^*$ obtained through an analogue of \Eqref{BC.tilts},
\[
\Bigl(\frac\partial{\partial u}L_N^t(u,v,w),
\frac\partial{\partial v}L_N^t(u,v,w)\Bigr)\biggm|_{(u^*,v^*,0)}
=(x_N,y_N)\,.
\]
Observing that (where for a real $x$ we write $x^+\DF\max(x,0)$)
\[
L_N^t(u,v,w)\equiv
\sum_{j=1}^NL\Bl(\frac u{\sigma_N}+\frac{N+1-j}{N+1}\,\frac v{\sigma_N}+\Bl(\frac{[tN]+1-j}{N+1}\Br)^+\,\frac w{\sigma_N}\Br)\,,
\]
we immediately obtain, up to a small correction,
\[
\frac1N
\frac\partial{\partial w}L_N^t(u^*,v^*,w)\biggm|_{w=0}
\approx\int_0^1(t-x)^+L'\bl(u^*+(1-x)v^*\br)\,dx\,,
\]
and thus the (conditional) mean value of the macroscopic polymer at ``time'' $t$ is  (recall \Eqref{vfi-Y.relation}) 
\begin{equation}
\label{eq:mean.profile}
t\xiL+c\int_0^t(t-x)L'\bl(u^*+(1-x)v^*\br)\,dx\,.
\end{equation}
In particular, in the Gaussian case $\Phi(x)=\kappa x^2/2$, the mean rescaled profile \Eqref{mean.profile} becomes
\[
t^2(1-t)\xiR+t(1-t)^2\xiL\,.
\]

It is instructive to compare the previous results to their analogues for the interfaces. Of course, the non-trivial geometry of the mean profile as well as anomalous $\bfC^1$-smoothness of the trajectories (recall the comment to \Eqref{theN-tightness-bound} above) are due to the nature of semiflexible interaction and are not present for interfaces.

\section{Free energy of a confined polymer}
As an application of the above estimates, we turn now to a problem that has often been studied in the physics literature (see, \eg, \cite{YBG07} and references therein): Determine the free energy (per unit of macroscopic length) of a semiflexible polymer constrained to lie inside a tube of given radius. From the mathematical point of view, this is equivalent to studying the logarithmic asymptotics of the probability of the event $\bl\{\sup_{0\leq k\leq N+1}|\varphi_k| \leq \rho\br\}$, when $N$ is large enough.

Using the functional CLT, it would be sufficient to prove the corresponding claim for the limiting Gaussian process. This so-called small ball problem has been studied for the integrated Brownian motion in~\cite{KhSh98}. We are going to give a completely different proof, in the spirit of~\cite{HrVe04}, which is easy and more robust, and also holds for positive values of $\eps$.
\begin{theorem}\sl
Let $c>0$ be the macroscopic length of the polymer. There exist constants $\rho_0=\rho_0(c)>0$, $C_1>0$, $C_2<\infty$ and $\delta>0$ such that, for all $\rho<\rho_0$,
$$
\frac{C_1}{\rho^{2/3}c^{1/3}} \leq -\frac1{c}\log\sfP_N\Bl(\sup_{1\leq k \leq N} |\varphi_k| \leq \rho \sigma_N\sqrt{N} \,|\, \varphi_0=\varphi_1=0\Br) \leq \frac{C_2}{\rho^{2/3}c^{1/3}}\,,
$$
uniformly in $\eps<\delta\rho^{2/3}c^{1/3}$.
\end{theorem}
\begin{remark}\rm
\begin{enumerate}
\item The existence of the limit as $\eps\equiv c/N\to 0$ can be proved using a standard subadditivity argument, see~\cite{KhSh98}. An explicit expression for the limit seems to be unknown (although the physicists have good numerical estimates).
\item A similar result holds for other  boundary conditions,
as long as $\varphi_0$ is not chosen too close to the boundary of the tube, and $\xi_1$ is small enough. A similar remark applies for $\varphi_{N+1}$ and $\xi_{N+1}$ (which were unconstrained above). For example, the proof remains unchanged if the boundary conditions at both extremities satisfy the same constraints as demanded by the event $\calA$ in the proof.

\item Although the above expression might look superficially different from the one given by the physicists' derivations, they actually coincide. To see this, it is best to restrict attention to the case studied in the physics literature, in which the Hamiltonian is of the form $\tfrac\kappa{2\epsilon} \sum_{i=1}^N (\Delta\varphi_i)^2$, and to write down explicitly the temperature dependence. In that case, $\sigma^2_N=N/(\beta\kappa c)$, where $\beta=1/(k_{\rm B}T)$ is the inverse temperature. To match the physicists' procedure, we wish to measure the width of the tube in units set by the polymer length. The event we are interested in thus becomes
\[
\sup_{1\leq k \leq N} N^{-3/2} |\varphi_k| \leq r/c\,,
\]
where we have denoted by $r=\sqrt{c/(\beta\kappa)}\rho$ the macroscopic width of the tube. We then see that the free energy is given by $k_{\rm B}T\,  (\beta\kappa)^{-1/3} r^{-2/3}$, which agrees perfectly with the physicists' expression, since $\beta\kappa$ is the persistence length corresponding to these parameters.
\end{enumerate}
\end{remark}
\begin{proof}
\emph{Lower bound on the probability.}
We write
\[
 R=\rho\sigma_N\sqrt{N}\qquad\text{ and }\qquad D=[ \eps^{-2/3} \sigma_N^{-2/3} R^{2/3} ] = [\rho^{2/3} c^{1/3} \eps^{-1}]\,.
\]
Let also $\nu>0$ be a small number (to be chosen below) and denote by $\mathcal{A}$ the event that
\begin{itemize}
\item $|\varphi_{kD+1}| \leq \nu R$, for all $1\leq k\leq [ N/D ]$;
\item $|\xi_{kD+1}| \leq \nu R/D$, for all $1\leq k\leq [ N/D ]$.
\end{itemize}
We then have the lower bound
\begin{multline*}
\sfP_N\bl(\sup_{1\leq k \leq N} |\varphi_k| \leq R \,|\, \varphi_0=\varphi_1=0\br)\\
\geq \sfP_N\bl(\sup_{1\leq k \leq N} |\varphi_k| \leq R \,|\, \varphi_0=\varphi_1=0,\mathcal{A}\br)\, \sfP_N(\mathcal{A} \,|\, \varphi_0=\varphi_1=0).
\end{multline*}
Let us first find a lower bound for $\sfP_N(\mathcal{A}\,|\, \varphi_0=\varphi_1=0)$. Conditioning on the pairs $\varphi_{kD},\varphi_{kD+1}$, $1\leq k\leq [ N/D]$ (compatible with the event $\mathcal{A}$), the Markov property implies that it is sufficient to consider what happens in a single piece $\{(k-1)D,\ldots,kD+1\}$. Namely, for $|a_0|\leq\nu$ and $|g_0|\leq \nu$, it is enough to prove that
$$
\sfP_D\bl(|\varphi_{D+1}| \leq \nu R, |\xi_{D+1}| \leq \nu R/D \,|\, \varphi_0=a_0 R, \xi_1=g_0 R/D\br)
$$
is bounded away from zero, uniformly in $c,\rho$ and $\eps$.
Rewriting this event in terms of the random variables $X_D$ and $Y_D$ yields
\begin{align*}
\sfP_D \Bigl( \frac 1{\sigma_N\sqrt{D+1}}\, Y_D \in& \Bl[ -\frac{(\nu+a_0+g_0)R}{\eps D\sigma_N\sqrt{D+1}}, \frac{(\nu-a_0-g_0)R}{\eps D\sigma_N\sqrt{D+1}} \Br],\\
\frac 1{\sigma_N\sqrt{D+1}}\,& X_D \in \Bl[ -\frac{(\nu+g_0)R}{\eps D\sigma_N\sqrt{D+1}}, \frac{(\nu-g_0)R}{\eps D\sigma_N\sqrt{D+1}} \Br] \Bigr).
\end{align*}
Since
\begin{equation}
\label{equ_asympt}
\frac R{\eps D\sigma_N\sqrt{D+1}} = 1+o(1),\qquad\text{as $D\to\infty$,}
\end{equation}
the Central Limit Theorem~\ref{thm:CLT} implies that the above probability converges, as $D\to\infty$, to
\begin{multline*}
\sfP\Bigl( Z_1 \in \bigl[ -(\nu+a_0+g_0), \nu-a_0-g_0 \bigr], Z_0 \in \bigl[ -(\nu+g_0), \nu-g_0 \bigr] \Bigr)\\
\geq \sfP\Bigl( Z_1 \in \bigl[ -3\nu, -\nu \bigr], Z_0 \in \bigl[ -2\nu, 0 \bigr] \Bigr),
\end{multline*}
where $(Z_0,Z_1)$ is a Gaussian vector with zero mean and covariance matrix
\[
 \sfQ^0=\begin{pmatrix} 1 & 1/2 \\ 1/2 & 1/3\end{pmatrix}\,.
\]
This probability being bounded away from zero, uniformly in $\eps$, $\rho$ and $c$, we conclude that
$$
\sfP_N(\mathcal{A}\,|\, \varphi_0=\varphi_1=0) \geq e^{-C\, N/D} = e^{-C\,\rho^{-2/3}c^{2/3}},
$$
uniformly in $\eps,\rho,c$ such that $D\approx \rho^{2/3} c^{1/3} \eps^{-1}$ is sufficiently large.

Let us now turn to the derivation of a lower bound on
$$
\sfP_N\bl(\sup_{1\leq k \leq N} |\varphi_k| \leq R \,|\, \varphi_0=\varphi_1=0,\mathcal{A}\br).
$$
For $|a_0|,|a_{D+1}|\leq\nu$ and $|g_1|,|g_{D+1}|\leq \nu$, let us introduce the event
\begin{multline*}
\calB = \calB(a_0,a_{D+1},g_1,g_{D+1}) = \Bl\{\varphi_0=a_0 R,\varphi_{D+1}=a_{D+1} R,\\\xi_1=g_1 R/D,\xi_{D+1}=g_{D+1}R/D\Br\}.
\end{multline*}
Changing to the $X,Y$ variables yields,
$$
\sfP_D\bigl(\sup_{1\leq k \leq D} |\varphi_k| \geq R \,\bigm|\, \calB \bigr)
\leq
\sfP\Bl(\sup_{1\leq k \leq D} \frac{|Y_k|}{\sigma_N\sqrt{D+1}} \geq \frac{(1-2\nu)R}{\eps D\sigma_N\sqrt{D+1}} \,\bigm|\,  \calB\Br).
$$
Fixing some $\nu<\tfrac14$, the functional CLT and~\eqref{equ_asympt} then imply that, for all $D$ large enough, the latter probability is bounded above by
$$
\sfP\Bigl(\sup_{t\in[0,1]} |\theta(t)| \geq \tfrac 13\Bigr),
$$
where $\theta(t)$ is the Gaussian process characterized by~\eqref{eq:theta-moments} with $a=g_{D+1}-g_1$ and $b=a_{D+1}-a_0-g_1$. An application of Fernique's inequality~\cite{Fe64} shows that this probability is bounded above uniformly in $a_0,g_1,a_{D+1},g_{D+1}$ in the range considered.
The Markov property then implies that
$$
\sfP_N\Bl(\sup_{1\leq k \leq N} |\varphi_k| \leq R \,|\, \varphi_0=\varphi_1=0,\mathcal{A}\Br) \geq e^{-C\,\rho^{-2/3}c^{2/3}},
$$
uniformly in $\eps,c,\rho$ such that $D\approx \rho^{2/3} c^{1/3} \eps^{-1}$ is sufficiently large. This completes the proof of the lower bound.

\bigskip
\emph{Upper bound on the probability.}
As for the lower bound, we partition the tube into disjoint pieces of length $D=[\eps^{-2/3}\sigma_N^{-2/3}R^{2/3}]$. We then write
\begin{multline*}
\sfP_N\bl(\sup_{1\leq k \leq N} |\varphi_k| \leq R \,|\, \varphi_0=\varphi_1=0\br)\\
\le
\prod_{i=1}^{[ N/D ]} \sfP_N\Bl(\!\sup_{(i-1)D+2\leq k \leq iD+1}\! |\varphi_k| \leq R \,\bigm|\, \varphi_0=\varphi_1=0,\!\sup_{2\leq k \leq (i-1)D+1}\! |\varphi_k| \leq R\Br).
\end{multline*}
(If $N/D$ is not an integer, we simply bound the contribution of the last, shorter, piece by $1$.)
We are going to show that each of the remaining terms in the product is bounded away from $1$, uniformly in $\eps,c,\rho$, provided $D$ is large enough. The conclusion will then clearly follow.

Using once more the Markov property, we see that it suffices to bound
\begin{align*}
\sup_{\substack{|a_0|\leq 1\\g_1}} \sfP_{D}\bl(\sup_{2\leq k \leq D+1} &|\varphi_k| \leq R \,|\, \varphi_0=a_0 R, \xi_1=g_1 R/D\br)\\
&=
1-\inf_{\substack{|a_0|\leq 1\\g_1}} \sfP_{D}\bl(\sup_{2\leq k \leq D+1} |\varphi_k| > R \,|\, \varphi_0=a_0 R, \xi_1=g_1 R/D\br)\\
&\leq
1-\inf_{\substack{|a_0|\leq 1\\g_1}} \sfP_{D}\bl(|\varphi_{D}| > R \,|\, \varphi_0=a_0 R, \xi_1=g_1 R/D\br).
\end{align*}
We shall now separately deal with the cases $|g_1|\leq M$ and $|g_1|> M$, where $M$ is some large enough number which will be chosen below.

First,
$$
\inf_{\substack{|a_0|\leq 1\\|g_1|\leq M}} \sfP_{D}\bl(|\varphi_{D}| > R \,|\, \varphi_0=a_0 R, \xi_1=g_1 R/D\br)
$$
can be bounded below by
$$
\sfP\Bl(\frac{|Y_{D}|}{\sqrt{\sigma_N^2 (D+1)}} > \frac{(M+2) R}{\eps D \sigma_N\sqrt{D+1}}\Br),
$$
and the Central Limit Theorem~\ref{thm:CLT} and~\eqref{equ_asympt} imply that the latter converges, as $D\to\infty$, to $\sfP(|Z_1| > M+2)$, which is bounded away from $0$ by a constant depending only on $M$.

Second, straightforward computations similar to those done in the proof of the CLT yield
\begin{align*}
\sfE_{D} \bigl(\varphi_{D} \,\bigm|\, \varphi_0=a_0 R, \xi_1=g_1 R/D \bigr)
&=
(a_0 + g_1 + o(1)) R,\\
\Var_{D}\bigl(\varphi_{D} \,\bigm|\, \varphi_0=a_0 R, \xi_1=g_1 R/D \bigr) &= \tfrac16 (D+1)D(2D+1) \sigma_N^2 \eps^2.
\end{align*}
We conclude that, when $|a_0|\leq 1$ and $|g_1|>M$,
$$
\bigl| \sfE_{D} \bigl(\varphi_{D} \,\bigm|\, \varphi_0=a_0 R, \xi_1=g_1 R/D \bigr) \bigr| \geq (M-2) R.
$$
Let us write $\psi_D=\varphi_D-\sfE_{D} \bigl(\varphi_{D} \,\bigm|\, \varphi_0=a_0 R, \xi_1=g_1 R/D \bigr)$. Chebychev's inequality implies that
\begin{align*}
\inf_{\substack{|a_0|\leq 1\\|g_1|>M}} \sfP_{D}&(|\varphi_{D}| \leq R \,|\, \varphi_0=a_0 R, \xi_1=g_1 R/D)\\
&\leq
\inf_{\substack{|a_0|\leq 1\\|g_1|>M}} \sfP_{D}(|\psi_{D} | \geq (M-3) R \,|\, \varphi_0=a_0 R, \xi_1=g_1 R/D)\\
&\leq
\frac{1}{3(M-3)^2},
\end{align*}
and the latter is smaller than $1/3$, provided $M\geq 4$.
\end{proof}

\appendix
\section{Heuristic derivation of the model}\label{sec:heuristic}

We wish to construct a discretized version of the {\it worm-like chain} model from Sect.~\ref{sec:model}. Given a positive $\eps$, we associate to the macroscopic polymer profile $f$ the discretized polymer configuration
\[
 \vfi_k\DF\eps^{-\gamma}f(k\eps);
\]
\ie, we discretize the polymer horizontally with step $\eps$ and vertically with step $\eps^\gamma$, where the (yet unknown) parameter $\gamma$ has to be determined. In order to determine $\gamma$, we proceed as follows. To each polymer configuration $\bvfi=(\vfi_0=0,\ldots,\vfi_{N+1})$, with $N=[c/\eps]$, we associate the energy
\begin{equation}
\label{eq:HamUndef}
\calH_N(\bvfi)\DF\eps\sum_{j=1}^{N}\Phi\bl(\eps^{-\delta}\Delta\vfi_j\br)\,.
\end{equation}
For a smooth profile $f$ we then have
\[
 \Delta\vfi_k\approx\eps^{2-\gamma} f''(k\eps),
\]
so that the macroscopic expression for the energy is recovered, in the limit $\eps\to 0$,
\begin{equation}
\label{eq:Hamiltonian-limit}
 \calH_{[c/\eps]}(\bvfi)\approx\eps\sum_{j=1}^{[c/\eps]}\Phi\bl(f''(j\eps)\br)\approx\int_0^c\Phi\bl(f''(x)\br)\,dx\,,
\end{equation}
provided the relation
\begin{equation}
\label{eq:scaling-relation}
\gamma+\delta=2\
\end{equation}
is verified.
The above computation holds for all $\gamma$, $\delta>0$ satisfying \Eqref{scaling-relation}. Here, we choose $\gamma=\delta=1$, so that for a sufficiently smooth profile $f(\cdot)$ we have $\nabla\vfi_k\approx f'(\eps k)$, \ie, the macroscopic and microscopic gradients coincide.
As shown in Sect.~\ref{sec:scaling}, for the class of models considered in the present paper this scaling results in both the vertical fluctuations and the end-to-end gradient-gradient fluctuations for such polymers being macroscopic. This, in particular, implies that the persistence length and the polymer length are of the same order.

\end{document}